\documentclass[reqno,a4paper,11pt]{amsart}
\usepackage{amsmath, amsthm, amscd, amsfonts, amssymb, graphicx, color}
\usepackage{amssymb}
\usepackage[english]{babel}
\usepackage[dvips, lmargin=4cm, rmargin=4cm, tmargin=4cm, bmargin=4cm]{geometry}
\usepackage[colorlinks=true,urlcolor=blue,linkcolor=blue]{hyperref}
\usepackage{hyperref}
\usepackage[T1]{fontenc}
\usepackage[english]{babel}

\usepackage{enumitem}
\usepackage{dsfont}
\newtheorem{thm}{Theorem}[section]
\newtheorem{defini}{Definition}[section]

\newtheorem{lem}{Lemma}[section]
\newtheorem{prop}{Proposition}[section]
\newtheorem{coro}{Corollary}[section]

\def \R {\mathbb{R} }

\everymath{\displaystyle}
\begin{document}
\title[ ]{On Fractional Anisotropic Musielak-Sobolev Spaces with Applications to Nonlocal Eigenvalue Problems}
\author[Mohammed SRATI]
{Mohammed SRATI}
\address{ Mohammed SRATI\newline
 Laboratory of Applied Mathematics of the Oriental Region (LaMAO), High School of Education and Formation (ESEF), University Mohammed First, Oujda,
 Morocco.}
\email{srati93@gmail.com}

\subjclass[2010]{35R11, 35P30, 35J20, 58E05.}
\keywords{Fractional anisotropic Musielak-Sobolev spaces, eigenvalue problems, variational Methods.}
\maketitle
\begin{abstract}
In this paper, we introduce and study a new class of fractional modular function spaces, called \emph{Fractional Anisotropic Musielak--Sobolev Spaces}, which generalize both the fractional Anisotropic Orlicz--Sobolev spaces and the Anisotropic fractional Sobolev spaces with variable exponent. These spaces are designed to handle anisotropic and heterogeneous behaviors that naturally arise in nonlocal and nonlinear models. We develop their fundamental properties and embedding results, establishing a solid variational framework. As an application, we investigate a class of nonlocal anisotropic eigenvalue problems involving variable growth and direction-dependent fractional integro-differential operators. We prove the existence of eigenvalues by means of critical point theory and modular analysis. Our results extend and unify several existing models in the theory of nonlocal partial differential equations.
\end{abstract}
\tableofcontents
\section{Introduction}

In recent years, fractional partial differential equations (PDEs) have gained increasing attention due to their ability to model various phenomena involving long-range interactions, memory effects, and anomalous diffusion processes. These equations provide a more accurate and flexible framework than classical local models in applications such as image processing, phase transitions, materials science, and finance. The nonlocal nature of fractional operators allows them to capture complex interactions that cannot be described by integer-order derivatives.

A particularly rich framework for studying such problems is provided by \emph{modular function spaces}, such as \emph{fractional Orlicz--Sobolev spaces} and their generalizations (see \cite{sr5, 3, srati3, SR, SRT, chine, SRH, bh1,sr_mo, sal1, sal2, sal3}). These spaces allow for the treatment of nonlinearities that go beyond the standard power-type growth, accommodating, for example, exponential-type behaviors. Among them, \emph{Musielak--Orlicz spaces} stand out due to their ability to model media with nonhomogeneous and spatially dependent characteristics. When extended to the \emph{fractional setting}, Musielak--Orlicz--Sobolev spaces enable the rigorous analysis of fractional operators with nonstandard growth and variable exponents, reflecting more faithfully the complexity of real-world phenomena.

Another fundamental aspect in modern analysis is \emph{anisotropy}, which appears naturally in models where the behavior of a system depends on direction. Anisotropic problems arise in many physical and engineering applications, such as crystal growth, porous media flow, and image reconstruction. From a mathematical point of view, anisotropic spaces introduce direction-dependent metrics or integrability conditions, making the analysis more delicate but also more representative of real anisotropic media.


    Anisotropy in mathematical analysis naturally arises in continuum mechanics, notably in the modeling of composite materials and heterogeneous media. It appears when physical properties depend on direction. This idea leads to differential operators whose coefficients or derivatives are not invariant under rotations — these are called anisotropic operators.\\
    
    In this context, a foundational contribution was made by Troisi in 1969 \cite{Troi}, who conducted a rigorous study of anisotropic operators within the framework of elliptic and parabolic equations with variable coefficients. These operators often take the form
    \[
    -\sum_{i=1}^N \partial_{x_i}\left(a_i(x)\, \partial_{x_i} u\right),
    \]
    where each spatial direction can exhibit distinct behavior, reflecting the anisotropic nature of the medium. To address such problems, Troisi introduced the notion of \textit{anisotropic Sobolev spaces}, defined as
    \[
    W^{1,\vec{p}}(\Omega) = \left\{ u \in L^{p_i}(\Omega) : \partial_{x_i} u \in L^{p_i}(\Omega) \text{ for each } i \right\},
    \]
    where \( \vec{p} = (p_1, \dots, p_N) \) is a vector of constant exponents, allowing directional dependence of regularity. These spaces have proven essential in analyzing PDEs where the smoothness of solutions may differ by axis.\\
    
    Building on this framework, a more recent development extended the theory to accommodate variable exponent growth in each direction. Specifically, the \textit{anisotropic variable exponent Sobolev space} \( W^{1, \vec{p}(\cdot)}(\Omega) \), where \( \vec{p}(\cdot) = \left(p_1(\cdot), \ldots, p_N(\cdot)\right) \), was introduced \cite{Fan}. This generalization enables the modeling of media with spatially heterogeneous anisotropic behavior. Correspondingly, a new differential operator appeared in the literature:
    \[
    \Delta_{\vec{p}(x)}(u) = \sum_{i=1}^N \partial_{x_i} \left( \left| \partial_{x_i} u \right|^{p_i(x)-2} \partial_{x_i} u \right),
    \]
    which is referred to as the \textit{anisotropic variable exponent} \( \vec{p}(\cdot) \)-\textit{Laplace operator}. This operator generalizes both the classical \( \Delta_{p(x)} \) and the anisotropic constant exponent \( \Delta_{\vec{p}} \) operators; see also \cite{miha1,miha2} and the references therein.\\
    
    More recently, Bahdouni et al in \cite{bahbah} and  Azroul et al in \cite{kamali} extended the concept of anisotropy to the framework of \textit{fractional Sobolev spaces with variable exponents}. They introduced a fractional version of the anisotropic variable exponent Laplacian, which captures both the nonlocal and directionally heterogeneous nature of certain physical phenomena. The corresponding operator is defined by
    \[
    \left[(-\Delta)_{\vec{p}(x, \cdot)}^s u\right](x) = p.v. \sum_{i=1}^N \int_{\mathbb{R}^N} \frac{|u(x)-u(y)|^{p_i(x, y)-2}(u(x)-u(y))}{|x-y|^{N+s p_i(x, y)}} \, dy,
    \]
   $\text{for all } x \in \mathbb{R}^N$ where \( \vec{p}(x,y) = (p_1(x,y), \dots, p_N(x,y)) \) denotes a family of variable exponents depending on both spatial variables \( x \) and \( y \), and \( s \in (0,1) \) is the fractional order. This operator can be seen as a \textit{nonlocal and anisotropic generalization} of the classical \( \vec{p}(\cdot) \)-Laplace operator, and provides a natural framework for studying nonlocal problems in anisotropic and heterogeneous media.\\
    
    In the same direction, we proposed a further extension of the aforementioned framework by developing a theory within the setting of fractional anisotropic Orlicz–Sobolev spaces \cite{srtanis}. This approach allows for even greater flexibility in modeling media with highly nonlinear and nonstandard growth conditions. The corresponding nonlocal integro-differential operator of elliptic type is defined by
    \[
    (-\Delta)^s_{\vec{a}(\cdot)} u(x) = 2 \lim_{\varepsilon \searrow 0} \sum_{i=1}^N \int_{\mathbb{R}^N \setminus B_\varepsilon(x)} a_i\left( \frac{|u(x)-u(y)|}{|x-y|^{s}} \right) \frac{u(x)-u(y)}{|x-y|^{s}} \frac{dy}{|x-y|^{N+s}},
    \]
    $\text{for all } x \in \mathbb{R}^N$ where each \( a_i \) is an Orlicz-type function associated with the \( i \)-th direction. This operator generalizes previous models by incorporating both \textit{anisotropy} and \textit{nonstandard growth} in a nonlocal framework, and opens the way for the analysis of a wide class of nonlinear integro-differential problems in anisotropic Orlicz-type fractional settings.


The aim of this paper is to study a class of nonlinear, nonlocal, anisotropic eigenvalue problems involving a generalized fractional operator acting in a newly introduced function space, the \emph{Fractional Anisotropic Musielak--Sobolev Space}. This space extends and unifies the existing frameworks of the \emph{Fractional Anisotropic Orlicz–Sobolev Space}, which we introduced in  \cite{srtanis}, and the \emph{Anisotropic Fractional Sobolev Spaces with Variable Exponent}, introduced by Azroul et al. in \cite{kamali,bahbah}, thereby enabling the treatment of a broader class of problems involving nonstandard growth conditions and directional anisotropy.

We consider the following nonlocal eigenvalue problem:
\[
(P_a)\left\{
\begin{aligned}
& M\left( \sum_{i=1}^{N} \int_{Q} \varPhi_{x,y}^i\left( \frac{|u(x) - u(y)|}{|x - y|^{s_i}} \right) \frac{dxdy}{|x - y|^N} \right) 
  \sum_{i=1}^{N} (-\Delta)^{s_i}_{a^i_{(x,\cdot)}} u(x) \\[1.5ex]
& = \lambda |u(x)|^{q(x)-2} u(x) \quad \text{in } \Omega, \\[1.5ex]
& u(x) = 0 \quad \text{in } \mathbb{R}^N \setminus \Omega,
\end{aligned}
\right.
\]
where $\Omega \subset \mathbb{R}^N$ ($N \geq 1$) is a bounded open set with Lipschitz boundary, and 
\[
Q = \mathbb{R}^{2N} \setminus (C\Omega \times C\Omega), \quad \text{with} \quad C\Omega := \mathbb{R}^N \setminus \Omega.
\]
The function $q: \overline{\Omega} \times \overline{\Omega} \to (1, +\infty)$ is a bounded continuous function representing a variable exponent, $\lambda > 0$ is a real parameter, and $s_i \in (0, 1)$ for each $i \in \{1, \ldots, N\}$.

The operators $(-\Delta)^{s_i}_{a^i_{(x,\cdot)}}$ appearing in this equation are fractional anisotropic nonlocal integro-differential operators of elliptic type, defined for every $x \in \mathbb{R}^N$ and $i \in \{1, \ldots, N\}$ by:
\[
(-\Delta)^{s_i}_{a^i_{(x,\cdot)}} u(x) = 2 \lim_{\varepsilon \searrow 0} \int_{\mathbb{R}^N \setminus B_\varepsilon(x)} a^i_{(x,y)}\left( \frac{|u(x) - u(y)|}{|x - y|^{s_i}} \right) \frac{u(x) - u(y)}{|x - y|^{s_i}} \frac{dy}{|x - y|^{N + s_i}},
\]
where each $a^i_{(x,y)} : \mathbb{R} \to \mathbb{R}$ is a nonlinear function to be specified later. These operators are designed to capture directional dependence (anisotropy), nonlocal interactions, and nonlinear constitutive laws via the structure of the kernel $a^i_{(x,y)}$.

The modular function $\varPhi_{x,y}^i$ models anisotropic and heterogeneous energy densities, potentially varying with position and direction,  generalizing classical $p$-Laplace and Orlicz-type behavior. The function $M: [0, \infty) \to [0, \infty)$ represents a nonlinearity that can govern the intensity of diffusion.

This problem setup allows for modeling phenomena in heterogeneous and anisotropic media, where the behavior of the solution is influenced by both the magnitude and direction of interactions at a distance, as well as position-dependent growth conditions.

The main contributions of this paper are threefold:
\begin{itemize}
  \item We define and analyze the structure of the Fractional Anisotropic Musielak--Sobolev Space, providing embedding theorems and modular inequalities appropriate to our setting.
  \item We give a precise variational formulation of problem $(P_a)$ and establish conditions under which weak solutions exist, based on direct methods in the calculus of variations and compactness arguments.
  \item We investigate spectral properties of the associated operator and discuss the behavior of the first eigenvalue in this generalized setting.
\end{itemize}

The structure of the paper is as follows: Section~2 lays the groundwork with mathematical
preliminaries, including definitions and key properties of fractional Musielak-Orlicz Sobolev spaces. In Section~3, we introduce the necessary background on fractional modular spaces and develop the theory of the Fractional Anisotropic Musielak--Sobolev Space, including several qualitative properties such as modular inequalities and embedding results. Section~4 is devoted to the variational formulation of the eigenvalue problem \hyperref[P]{$(P_a)$}, where we present our main results and, using variational techniques, establish the existence or nonexistence of eigenvalues.


\section{Variatoinal setting and preliminaries results}

To deal with this situation we define the fractional Musielak-Sobolev space to investigate Problem \hyperref[P]{$(P_a)$}. Let us recall the definitions and some elementary properties of this spaces. We refer the reader to \cite{benkirane,benkirane2} for further reference and for some of the proofs of the results in this section.\\

 Let $\Omega$ be an open subset of $\R^N$, $N\geqslant 1$. We assume that any $i=1,...,N,$    $(x,y,t)\mapsto a^i_{(x,y)}(t):=a^i(x,y,t) : \overline{\Omega}\times\overline{\Omega}\times \R\longrightarrow \R$   are symmetric functions :
  \begin{equation}\label{n4}
  a^i(x,y,t)=a^i(y,x,t) ~~ \forall(x,y,t)\in \overline{\Omega}\times\overline{\Omega}\times \R,\end{equation}
    and the functions : $\varphi^i(\cdot,\cdot,\cdot) : \overline{\Omega}\times\overline{\Omega}\times \R \longrightarrow \R$ $(i = 1,...,N)$ defined by  
 $$
   \varphi^i_{x,y}(t):=\varphi^i(x,y,t)= \left\{ 
           \begin{array}{clclc}
         a^i(x,y,|t|)t   & \text{ for }& t\neq 0, \\\\
           0  & \text{ for } & t=0,
           \end{array}
           \right. 
 $$
 are increasing homeomorphisms from $\R$ onto itself. For $i = 1,...,N$, let 
 $$\varPhi^i_{x,y}(t):=\varPhi^i(x,y,t)=\int_{0}^{t}\varphi^i_{x,y}(\tau)d\tau~~\text{ for all } (x,y)\in \overline{\Omega}\times\overline{\Omega},~~\text{ and all } t\geqslant 0.$$  
 Then, $\varPhi^i_{x,y}$ are a Musielak functions (see \cite{mu}).
 
 Also, we take $ \widehat{a}^i_x(t):=\widehat{a}^i(x,t)=a^i_{(x,x)}(t)  ~~ \forall~ (x,t)\in \overline{\Omega}\times \R$ $(i=1,...,N)$. Then the functions $\widehat{\varphi}^i(\cdot,\cdot) : \overline{\Omega}\times \R \longrightarrow \R$ defined  by :
   $$
      \widehat{\varphi}^i_{x}(t):=\widehat{\varphi}^i(x,t)= \left\{ 
           \begin{array}{clclc}
         \widehat{a}^i(x,|t|)t   & \text{ for }& t\neq 0, \\\\
           0  & \text{ for } & t=0,
           \end{array}
           \right. 
        $$
 are increasing homeomorphisms from $\R$ onto itself. If we set 
 \begin{equation}\label{phi}
 \widehat{\varPhi}^i_{x}(t):=\widehat{\varPhi}^i(x,t)=\int_{0}^{t}\widehat{\varphi}^i_{x}(\tau)d\tau ~~\text{ for all}~~ t\geqslant 0.
 \end{equation}  
 Then, $\widehat{\varPhi}^i_{x}$ is also a Musielak function.\\

 For the function $\widehat{\varPhi}_x$ $(i=1,...,N)$ given in (\ref{phi}), we introduce the Musielak space as follows
  $$L_{\widehat{\varPhi}^i_x} (\Omega)=\left\lbrace u : \Omega \longrightarrow \R \text{ mesurable }: \int_\Omega\widehat{\varPhi}^i_x(\lambda |u(x)|)dx < \infty \text{ for some } \lambda>0 \right\rbrace. $$
The space $L_{\widehat{\varPhi}^i_x} (\Omega)$ is a Banach space endowed with the Luxemburg norm 
$$||u||_{\widehat{\varPhi}^i_x}=\inf\left\lbrace \lambda>0 \text{ : }\int_\Omega\widehat{\varPhi}^i_x\left( \dfrac{|u(x)|}{\lambda}\right) dx\leqslant 1\right\rbrace. $$
 The conjugate function of $\varPhi^i_{x,y}$ $(i=1,...,N)$ is defined by $\overline{\varPhi^i}_{x,y}(t)=\int_{0}^{t}\overline{\varphi^i}_{x,y}(\tau)d\tau$ $\text{ for all } (x,y)\in\overline{\Omega}\times\overline{\Omega}$  $\text{ and all } t\geqslant 0$, where $\overline{\varphi^i}_{x,y} : \R\longrightarrow \R$ is given by $\overline{\varphi^i}_{x,y}(t):=\overline{\varphi^i}(x,y,t)=\sup\left\lbrace \alpha \text{ : } \varphi^i(x,y,\alpha)\leqslant t\right\rbrace.$ Furthermore, we have the following H\"older type inequality
  \begin{equation}
   \left| \int_{\Omega}uvdx\right| \leqslant 2||u||_{\widehat{\varPhi}^i_x}||v||_{\overline{\widehat{\varPhi}^i}_x}\hspace*{0.5cm} \text{ for all } u \in L_{\widehat{\varPhi}^i_x}(\Omega)  \text{ and } v\in L_{\overline{\widehat{\varPhi}^i}_x}(\Omega).
   \end{equation}
    Throughout this paper, for $i=1,...,N$, we assume that there exist two positive constants ${\varphi_i}^+$ and $\varphi_i^-$ such that 
\begin{equation}\label{v1}\tag{$\varPhi_1$}
    1<\varphi_i^-\leqslant\dfrac{t\varphi^i_{x,y}(t)}{\varPhi^i_{x,y}(t)}\leqslant \varphi_i^+<+\infty\text{ for all } (x,y)\in\overline{\Omega}\times\overline{\Omega}~~\text{ and all } t\geqslant 0. \end{equation}
    This relation implies  that
    \begin{equation}\label{A2}
        1<\varphi_i^-\leqslant \dfrac{t\widehat{\varphi}^i_{x}(t)}{\widehat{\varPhi}^i_{x}(t)}\leqslant\varphi_i^+<+\infty,\text{ for all } x\in\overline{\Omega}~~\text{ and all } t\geqslant 0.\end{equation}
             It follows that  $\varPhi^i_{x,y}$ and $\widehat{\varPhi}^i_{x}$ satisfy the global $\Delta_2$-condition (see \cite{ra}), written $\varPhi^i_{x,y}\in \Delta_2$ and $\widehat{\varPhi}^i_{x}\in \Delta_2$, that is,
    \begin{equation}\label{r1}
    \varPhi^i_{x,y}(2t)\leqslant K_1\varPhi^i_{x,y}(t)~~ \text{ for all } (x,y)\in\overline{\Omega}\times\overline{\Omega},~~\text{ and  all } t\geqslant 0,
    \end{equation} and
    \begin{equation}\label{rr1}
        \widehat{\varPhi}^i_{x}(2t)\leqslant K_2\widehat{\varPhi}^i_{x}(t) ~~\text{ for any } x\in\overline{\Omega},~~\text{ and  all } t\geqslant 0,
        \end{equation}
 where $K_1$ and $K_2$ are two positive constants. 
 
 Furthermore, for $i=1,...,N$, we assume that $\varPhi^i_{x,y}$ satisfies the following condition
  \begin{equation}\label{f2.}\tag{$\varPhi_2$}
  \text{ the function } [0, \infty) \ni t\mapsto \varPhi^i_{x,y}(\sqrt{t}) \text{ is convex. }
  \end{equation}
 
   \begin{lem}\cite{benkirane2}$\label{2.2..}$ Assume that \hyperref[v1]{$(\varPhi_1)$} is satisfied. Then for $i=1,...,N$, the following inequalities hold true:
                     \begin{equation}\label{3.}
   \varPhi^i_{x,y}(\sigma t)\geqslant \sigma^{\varphi_i^-}\varPhi^i_{x,y}(t) ~~\text{ for all } t>0  \text{ and any  } \sigma>1,
                     \end{equation}
        \begin{equation}\label{3.2}
   \varPhi^i_{x,y}(\sigma t)\geqslant \sigma^{\varphi_i^+}\varPhi^i_{x,y}(t) ~~\text{ for all }  t>0  \text{ and any } \sigma\in (0,1),
                     \end{equation}
    \begin{equation}\label{r10}
                        \varPhi^i_{x,y}(\sigma t)\leqslant \sigma^{\varphi_i^+}\varPhi^i_{x,y}(t) ~~\text{ for all } t>0 \text{ and any } \sigma>1,
                        \end{equation} 
       \begin{equation}\label{r11}
                       \varPhi^i_{x,y}(t)\leqslant \sigma^{\varphi_i^-}\varPhi^i_{x,y}\left( \dfrac{t}{\sigma} \right) ~~\text{ for all } t>0 \text{ and any } \sigma \in(0,1).
                       \end{equation}                                  
                     \end{lem} 
    
   Now, due to the nonlocality of the operators $(-\Delta)^{s_i}_{a_{(x,\cdot)}}$ $(i = 1,...,N)$,  we  define the fractional Musielak-Sobolev spaces as introduce in \cite{benkirane} as follows 
    \begingroup\makeatletter\def\f@size{9}\check@mathfonts$$ W^{s_i}{L_{\varPhi^i_{x,y}}}(\Omega)=\Bigg\{u\in L_{\widehat{\varPhi^i}_x}(\Omega) :  \int_{\Omega} \int_{\Omega} \varPhi^i_{x,y}\left( \dfrac{\lambda| u(x)- u(y)|}{|x-y|^{s_i}}\right) \dfrac{dxdy}{|x-y|^N}< \infty \text{ for some } \lambda >0 \Bigg\}.
$$\endgroup
This space can be equipped with the norm
\begin{equation}\label{r2}
||u||_{s_i,\varPhi^i_{x,y}}=||u||_{\widehat{\varPhi}^i_x}+[u]_{s_i,\varPhi^i_{x,y}},
\end{equation}
where $[\cdot]_{s_i,\varPhi^i_{x,y}}$ is the Gagliardo seminorm defined by 
$$[u]_{s_i(x,y),\varPhi^i_{x,y}}=\inf \Bigg\{\lambda >0 :  \int_{\Omega} \int_{\Omega} \varPhi^i_{x,y}\left( \dfrac{|u(x)- u(y)|}{\lambda|x-y|^{s_i}}\right) \dfrac{dxdy}{|x-y|^N}\leqslant 1 \Bigg\}.
$$

 \begin{thm}$($\cite{benkirane}$)$.
       Let $\Omega$ be an open subset of $\R^N$. For $i=1,...,N$, the spaces $W^{s_i}L_{\varPhi^i_{x,y}}(\Omega)$  are a Banach spaces with respect to the norm $(\ref{r2})$, and a  separable $($resp. reflexive$)$ spaces if and only if $\varPhi^i_{x,y} \in \Delta_2$ $($resp. $\varPhi^i_{x,y}\in \Delta_2 $ and $\overline{\varPhi^i}_{x,y}\in \Delta_2$$)$. Furthermore,
 if   $\varPhi^i_{x,y} \in \Delta_2$ and $\varPhi^i_{x,y}(\sqrt{t})$ is convex, then  the spaces $W^{s_i}L_{\varPhi^i_{x,y}}(\Omega)$ are an uniformly convex space.
       \end{thm}

            \begin{defini}$($\cite{benkirane}$)$.
            We say that $\varPhi^i_{x,y}$ $(i=1,...,N)$ satisfies the fractional boundedness condition, written $\varPhi^i_{x,y}\in \mathcal{B}_{f}$, if
          \begin{equation}\tag{$\varPhi_3$}
         \label{v3}         
            \sup\limits_{(x,y)\in \overline{\Omega}\times\overline{\Omega}}\varPhi^i_{x,y}(1)<\infty.  \end{equation}
            \end{defini}
            \begin{thm}  $($\cite{benkirane}$)$.    \label{TT}
                        Let $\Omega$ be an open subset of $\R^N$. Assume that  $\varPhi^i_{x,y}\in \mathcal{B}_{f}$ $(i=1,2)$. 
                        Then,
                        $$C^2_0(\Omega)\subset W^{s_i}L_{\varPhi^i_{x,y}}(\Omega).$$
                   \end{thm}

    
                
    In this paper, we assume that there exists $j\in\left\lbrace 1,...,N\right\rbrace $ such that:
            \begin{equation}\label{15}
            \int_{0}^{1} \dfrac{\widehat{\varPhi^j_x}^{-1}(\tau)}{\tau^{\frac{N+s_j}{N}}}d\tau<\infty,
           ~~~\text{and }~~
            \int_{1}^{\infty} \dfrac{\widehat{\varPhi^j_{x}}^{-1}(\tau)}{\tau^{\frac{N+s_j}{N}}}d\tau=\infty ~~\text{ for all }x\in \overline{\Omega}.
            \end{equation}
            We define the inverse  Musielak conjugate function of $\widehat{\varPhi^j_x}$ as follows
            \begin{equation}\label{17}
            (\widehat{\varPhi}^*_{j})^{-1}(t)=\int_{0}^{t}\dfrac{\widehat{\varPhi^j}_{x}^{-1}(\tau)}{\tau^{\frac{N+s_j}{N}}}d\tau.
            \end{equation}
             \begin{thm}\cite{benkirane2}\label{th2.}
           Let $\Omega$  be a bounded open
            subset of  $\R^N$ with $C^{0,1}$-regularity 
              and bounded boundary. If $(\ref{15})$  hold, then 
           \begin{equation}\label{18}
            W^{s_j}{L_{\varPhi^j_{x,y}}}(\Omega)\hookrightarrow L_ {\widehat{\varPhi}^*_{j}}(\Omega).
           \end{equation}
          Moreover, the embedding
                     \begin{equation}\label{27}
                      W^{s_j}{L_{\varPhi^j_{x,y}}}(\Omega)\hookrightarrow L_{B_x}(\Omega),
                     \end{equation}
                     is compact for all $B_x\prec\prec \widehat{\varPhi}^*_{j}.$
                     \end{thm}  
                     
  Next, we recall some useful properties of variable exponent spaces. For more details we refer the reader to \cite{23,27}, and the references therein.\\ 
    Consider the set
     $$C_+(\overline{\Omega})=\left\lbrace q\in C(\overline{\Omega}): q(x)>1\text{ for all } x \in\overline{\Omega}\right\rbrace .$$
     For all $q\in C_+(\overline{\Omega}) $, we define $$q^{+}= \underset{x\in \overline{\Omega}}{\sup}~q(x) \quad\text{and}\quad q^{-}= \underset{x\in \overline{\Omega}}{\inf}~q(x).$$
  For any  $q\in C_+(\overline{\Omega}) $, we define the variable exponent Lebesgue space as $$L^{q(x)}(\Omega)=\bigg\{u:\Omega\longrightarrow \mathbb{R} ~~\text{measurable}: \int_{\Omega}|u(x)|^{q(x)}dx<+\infty
  \bigg\}.$$
  This vector space endowed with the \textit{Luxemburg norm}, which is defined by
  $$\|u\|_{q(x)}= \inf \bigg\{\lambda>0:\int_{\Omega}\bigg|\frac{u(x)}{\lambda}\bigg|^{q(x)}dx \leqslant1 \bigg\}$$
  is a separable reflexive Banach space.
  
   A very important role in manipulating the generalized Lebesgue spaces with variable exponent is played by the modular of the $L^{q(x)}(\Omega)$ space, which defined by
   $$\begin{array}{clc}
   \hspace{-0.8cm}\rho_{q(.)}: L^{q(x)}(\Omega)\longrightarrow\mathbb{R}\\
    \hspace{5.3cm}u\longmapsto\rho_{q(.)}(u)=\displaystyle\int_{\Omega}|u(x)|^{q(x)}dx.
   \end{array}$$
  \begin{prop}$\label{anproop5}$
  Let $u\in  L^{q(x)}(\Omega) $, then we have
  \begin{enumerate}[label=(\roman*)]
  \item $\|u\|_{L^{q(x)}(\Omega)}<1$ $(resp. =1, >1)$ $\Leftrightarrow$ $ \rho_{q(.)}(u)<1$ $(resp. =1, >1)$,
  \item  $\|u\|_{L^{q(x)}(\Omega)}<1$ $\Rightarrow$ $\|u\|^{q{+}}_{L^{q(x)}(\Omega)}\leqslant \rho_{q(.)}(u)\leqslant \|u\|^{q{-}}_{L^{q(x)}(\Omega)}$,
  \item  $\|u\|_{L^{q(x)}(\Omega)}>1$ $\Rightarrow$ $\|u\|^{q{-}}_{L^{q(x)}(\Omega)}\leqslant \rho_{q(.)}(u)\leqslant \|u\|^{q{+}}_{L^{q(x)}(\Omega)}$.
  \end{enumerate}
  \end{prop}                            
  \section{Some qualitative properties of fractional Anisotropic Musielak-Sobolev spaces}
In order to study Problem \hyperref[P]{$(P_a)$}, it is important to encode the boundary condition $u=0$ in $\R^N\setminus \Omega$  in the weak formulation. In the case of fractional Sobolev space with variable exponent, Azroul et al \cite{SRH} introduced a new function space to study the variational functionals related to the fractional $p(x,.)$-Laplacian operator by observing the interaction between $\Omega$ and $\R^N\setminus \Omega$. Motivated by the above paper, and due to the nonlocality of the operator $(-\Delta)^s_{a_i(.)}$, we introduce the new fractional Orlicz-Sobolev space 
as follows
\begingroup\makeatletter\def\f@size{9}\check@mathfonts $$W^{s_i}L_{\varPhi_{x,y}^i}(Q)=\Bigg\{u\in L_{\widehat{\varPhi}_x^i}(\Omega) ~ :~ \int_{Q}  \varPhi_{x,y}^i\left( \dfrac{\lambda|u(x)- u(y)|}{|x-y|^{s_i}}\right) \dfrac{dxdy}{|x-y|^N}< \infty~~ \text{
 for some }\lambda>0 \Bigg\},
$$\endgroup
for any $i=1,...,N$, where $Q=\R^{2N}\setminus (C\Omega\times C\Omega)$ with $C\Omega=\R^N \setminus \Omega$. This spaces are equipped with the norm,
\begin{equation}\label{an6}
||u||_{i}=||u||_{\widehat{\varPhi}^i_x}+[u]_{i},
\end{equation}
where $[.]_{i}$ is the Gagliardo seminorm, defined by 
$$[u]_{i}=\inf \Bigg\{\lambda > 0 :  \int_{Q} \varPhi_{x,y}^i\left( \dfrac{|u(x)- u(y)|}{\lambda|x-y|^{s_i}}\right) \dfrac{dxdy}{|x-y|^N}\leqslant 1 \Bigg\}.
$$ 
 Similar to the spaces $(W^{s_i}L_{\varPhi_{x,y}^i}(\Omega), \|.\|_{s,\varPhi_i})$ we have that $(W^{s_i}L_{\varPhi_{x,y}^i}(Q),  \|.\|_{i})$ are a separable reflexive Banach spaces.\\
 
 Now, let $W_0^{s_i}L_{\varPhi^i_{x,y}}(Q)$ denotes the following linear subspace of $W^{s_i}L_{\varPhi^i_{x,y}}(Q),$
 $$W_0^{s_i}L_{\varPhi^i_{x,y}}(Q)=\left\lbrace u\in W^{s_i}L_{\varPhi_{x,y}^i}(Q) ~:~ u=0 \text{ a.e in } \R^N \setminus \Omega\right\rbrace $$
 with the norm
 $$[u]_{i}=\inf \Bigg\{\lambda > 0 :  \int_{Q} \varPhi^i_{x,y}\left( \dfrac{|u(x)- u(y)|}{\lambda|x-y|^{s_i}}\right) \dfrac{dxdy}{|x-y|^N}\leqslant 1 \Bigg\}.
 $$ 
 It is easy to check that $[u]_{i}$ is a norm on $W_0^{s_i}L_{\varPhi^i_{x,y}}(Q)$ (see Corollary $\ref{ann}$).
 
 In the following theorem, we compare the spaces $W_0^{s_i}L_{\varPhi^i_{x,y}}(\Omega)$ and $W_0^{s_i}L_{\varPhi^i_{x,y}}(Q)$.

 \begin{thm}\label{an2} For any $i=1,...,N$, the following assertions hold:
 \begin{itemize}
 \item[1)] The continuous embedding $$W^{s_i}L_{\varPhi^i_{x,y}}(Q)\subset W^{s_i}L_{\varPhi^i_{x,y}}(\Omega)$$
  holds true.\\
 \item[2)] If $u\in W_0^{s_i}L_{\varPhi^i_{x,y}}(Q)$, then $u\in W^{s_i}L_{\varPhi^i_{x,y}}(\R^N)$ and 
  $$||u||_{s_i,\varPhi_{x,y}^i}\leqslant ||u||_{W^{s_i}L_{\varPhi_{x,y}^i}(\R^N)}=||u||_{i}.$$
   \end{itemize}
   \end{thm}
   \begin{proof}[\textbf{Proof}]
   $1)$ Let $u\in W^{s_i}L_{\varPhi_{x,y}^i}(Q)$, since $\Omega\times \Omega\subsetneq Q,$ then for all $\lambda>0$ we have 
    \begin{equation}\label{an1}
   \int_{\Omega}\int_{\Omega} \varPhi_{x,y}^i\left( \dfrac{|u(x)- u(y)|}{\lambda|x-y|^{s_i}}\right) \dfrac{dxdy}{|x-y|^N}\leqslant \int_{Q} \varPhi_i\left( \dfrac{|u(x)- u(y)|}{\lambda|x-y|^{s_i}}\right) \dfrac{dxdy}{|x-y|^N}.    \end{equation}
    We set 
    $$\mathcal{A}^{s_i}_{\lambda,\Omega\times \Omega}=\Bigg\{\lambda > 0 :  \int_{\Omega}\int_{\Omega} \varPhi_{x,y}^i\left( \dfrac{|u(x)- u(y)|}{\lambda|x-y|^{s_i}}\right) \dfrac{dxdy}{|x-y|^N}\leqslant 1 \Bigg\}$$
    and 
     $$\mathcal{A}^{s_i}_{\lambda,Q}=\Bigg\{\lambda > 0 :  \int_{Q} \varPhi_i\left( \dfrac{|u(x)- u(y)|}{\lambda|x-y|^{s_i}}\right) \dfrac{dxdy}{|x-y|^N}\leqslant 1 \Bigg\}.$$
     By $(\ref{an1})$, it is easy to see that $\mathcal{A}^{s_i}_{\lambda,Q}\subset \mathcal{A}^{s_i}_{\lambda,\Omega\times\Omega}$. Hence 
     \begin{equation}\label{an}
    [u]_{s_i,\varPhi_{x,y}^i}=\inf\limits_{\lambda>0}\mathcal{A}^{s_i}_{\lambda,\Omega\times\Omega}\leqslant[u]_{i}=\inf\limits_{\lambda>0}\mathcal{A}^{s_i}_{\lambda,Q}. \end{equation}
  Consequently, by definitions of the norms $\|u\|_{s_i,\varPhi_{x,y}^i}$ and $\|u\|_{i},$ we obtain
    $$ \|u\|_{s_i,\varPhi_{x,y}^i}\leqslant \|u\|_{i}<\infty.$$
   $2)$ Let $u\in W_0^{s_i}L_{\varPhi_{x,y}^i}(Q)$, then $u=0$ in $\R^N\setminus \Omega$. So, $\|u\|_{L_{\widehat{\varPhi}^i_x}(\Omega)}=\|u\|_{\widehat{\varPhi}^i_x(\R^N)}.$ Since 
    $$\int_{\R^{2N}} \varPhi_{x,y}^i\left( \dfrac{|u(x)- u(y)|}{\lambda|x-y|^{s_i}}\right) \dfrac{dxdy}{|x-y|^N}=\int_{Q} \varPhi_{x,y}^i\left( \dfrac{|u(x)- u(y)|}{\lambda|x-y|^{s_i}}\right) \dfrac{dxdy}{|x-y|^N}$$
    for all $\lambda>0$. Then $[u]_{W^{s_i}L_{\varPhi_{x,y}^i}(\R^N)}=[u]_{i}$. Thus, we get
    
     $$||u||_{s_i,\varPhi_{x,y}^i}\leqslant ||u||_{W^{s_i}L_{\varPhi_{x,y}^i}(\R^N)}=||u||_{i}.$$
    \end{proof}
    \begin{coro}\label{ann}(Poincar\'{e} inequality)
    Let $\Omega$ be a bounded subset of $\R^N$. Then there exists a positive
    constant $c$ such that,
    $$
    \|u\|_{\widehat{\varPhi}_i^x}\leqslant c[u]_i, ~~~~\forall u\in W^{s_i}_0L_{\varPhi_{x,y}^i}(Q).$$
    \end{coro}
     \begin{proof}[\textbf{Proof}]
 Let $u\in W^{s_i}_0L_{\varPhi_{x,y}^i}(Q)$, by Theorem $\ref{an2}$, we have $u\in W^{s_i}_0L_{\varPhi_{x,y}^i}(\Omega)$. Then by \cite[Theorem 2.3]{benkirane2},  there exists a positive
     constant $c$ such that,
     $$
     \|u\|_{\widehat{\varPhi}_i^x}\leqslant c[u]_{s_i,\varPhi_{x,y}^i}.$$
    Combining the above inequality with $(\ref{an})$, we obtain that 
     $$
         \|u\|_{\widehat{\varPhi}_i^x}\leqslant c[u]_i, ~~~~\forall u\in W^{s_i}_0L_{\varPhi_{x,y}^i}(Q).$$
     \end{proof}    
 Now, we introduce a natural  fractional anisotropic Musielak-Sobolev spaces $W_0^{\overrightarrow{s}}L_{\overrightarrow{\varPhi}}(Q)$, that will enable us to study Problem \hyperref[P]{$(P_a)$}. For this purpose, let us denote by $\overrightarrow{\varPhi} : \Omega \longrightarrow \R^N$ the vectorial function $\overrightarrow{\varPhi}=(\varPhi_1,...,\varPhi_N)$.
 We define $W_0^{\overrightarrow{s}}L_{\overrightarrow{\varPhi}}(Q)$, the fractional anisotropic Musielak-Sobolev space as follows the closure of $C^\infty_0(\Omega)$ with respect to the norm:
 $$\|u\|_{\overrightarrow{\varPhi}}=\sum_{i=1}^{N}[u]_{i}.$$
Denoting $X=L_{\widehat{\varPhi}_x^1}(\Omega)\times...\times L_{\widehat{\varPhi}_x^N}(\Omega)$ and considering  the operator $T : W_0^{\overrightarrow{s}}L_{\overrightarrow{\varPhi}}(Q)\longrightarrow X,$ defined by $T(u)=\left( D^{s_1}u,...,D^{s_N}u\right)$
where $$D^{s_i}u=\dfrac{u(x)-u(y)}{|x-y|^{s_i}}.$$ It is clear that $W_0^{\overrightarrow{s}}L_{\overrightarrow{\varPhi}}(Q)$ and $X$ are isometric by $T$, since $$\|T(u)\|_{X}=\sum_{i=1}^{N}[u]_{i}=\|u\|_{\overrightarrow{\varPhi}}.$$
Thus $T(W_0^{\overrightarrow{s}}L_{\overrightarrow{\varPhi}}(Q))$ is a closed subspace of $X$, which is a reflexive Banach space. By  \cite[Proposition III.17]{breziz}, it follows that  $T(W_0^{\overrightarrow{s}}L_{\overrightarrow{\varPhi}}(Q))$ is reflexive and consequently  $W_0^{\overrightarrow{s}}L_{\overrightarrow{\varPhi}}(Q)$ is also reflexive Banach space.
 
On the other hand, in order to facilitate the manipulation of the space $W_0^{\overrightarrow{s}}L_{\overrightarrow{\varPhi}}(Q)$, we introduce $\overrightarrow{\varphi}^+, \overrightarrow{\varphi}^-\in \R^N$ as
$$ \overrightarrow{\varphi}^+=(\varphi_1^+,...,\varphi_N^+),~~ \text{and}~~ \overrightarrow{\varphi}^-=(\varphi_1^-,...,\varphi_N^-),$$
 and $\varphi^+_{max}$, $\varphi^-_{max}$, $\varphi^-_{min}$ as 
 $$ \varphi^+_{max}=\max\left\lbrace \varphi_1^+,...,\varphi_N^+\right\rbrace,~~ \varphi^-_{max}=\max\left\lbrace \varphi_1^-,...,\varphi_N^-\right\rbrace,$$ $$ \varphi^-_{min}=\min\left\lbrace \varphi_1^-,...,\varphi_N^-\right\rbrace.$$
 Throughout this paper we assume that 
 \begin{equation}\label{an8}
   \lim\limits_{t\rightarrow \infty}\dfrac{|t|^{q^+}}{(\varPhi_x^j)_*(kt)}=0 ~~\forall k>0,
   \end{equation}
   where $j\in\left\lbrace 1,...,N\right\rbrace $ is given in $(\ref{15})$.
 \begin{thm}\label{anth5}
   Let $\Omega$  be a bounded open
             subset of  $\R^N$ with $C^{0,1}$-regularity 
               and bounded boundary. Then 
             the embedding    
                    \begin{equation}\label{an27}
                     W_0^{\overrightarrow{s}}L_{\overrightarrow{\varPhi}}(Q)\hookrightarrow L^{q(x)}(\Omega)
                    \end{equation}
                    is compact.
                    \end{thm}
\begin{proof}[\textbf{Proof}]
Let $u\in  W_0^{\overrightarrow{s}}L_{\overrightarrow{\varPhi}}(Q)$, so $u\in  W_0^{s_i}{L_{\varPhi_{x,y}^i}}(Q)$ for any $i=1,...N$, and by Theorem $\ref{an2}$ we have $u\in  W_0^{s_i}{L_{\varPhi_{x,y}^i}}(\Omega)$ for any $i=1,...N$, then for $j\in \left\lbrace 1,...,N\right\rbrace $ given by ($\ref{an8}$), we can apply Theorem $\ref{th2.}$, and we have
$$\|u\|_{q^+}\leqslant c[u]_{s_i,\varPhi_{x,y}^j}\leqslant c[u]_{j}\leqslant c\sum_{i=1}^{N}[u]_{i}=c\|u\|_{\overrightarrow{\varPhi}}.$$
This implies that 
$$W_0^{\overrightarrow{s}}L_{\overrightarrow{\varPhi}}(Q)\hookrightarrow L^{q^+}(\Omega).$$
That fact combined with the continuous embedding of $L^{q^+}(\Omega)$ in $L^{q(x)}(\Omega)$  ensures that $W_0^{\overrightarrow{s}}L_{\overrightarrow{\varPhi}}(Q)$ is compactly embedded in $L^{q(x)}(\Omega)$.
\end{proof}
  We put $$\Psi(u)=\displaystyle\sum_{i=1}^{N}\int_{Q}\varPhi_{x,y}^i\left(\dfrac{|u(x)- u(y)|}{|x-y|^{s_i}}\right) \dfrac{dxdy}{|x-y|^N},$$                               
 \begin{prop}\label{norm}
 On $W_0^{\overrightarrow{s}}L_{\overrightarrow{\varPhi}}(Q)$ the following norm 
 $$\|u\|_{\overrightarrow{\varPhi}}=\sum_{i=1}^{N}[u]_{i},$$
 $$||u||_{\max}=\max_{i=1,...,N} [u]_{i}, $$
 $$||u||=\inf\left\lbrace \lambda>0 \text{ : } \Psi(u) \leqslant 1\right\rbrace, $$
  are equivalents.
   \end{prop} 
  \begin{proof}[\textbf{Proof}]
    First, we point out that $\|\cdot\|_{\overrightarrow{\varPhi}}$ and $||\cdot||_{max}$ are equivalent, since 
    \begin{equation}\label{eq0}
    N||u||_{max}\geqslant \|u\|_{\overrightarrow{\varPhi}}\geqslant ||u||_{max} \text{ for all } u\in W_0^{\overrightarrow{s}}L_{\overrightarrow{\varPhi}}(Q).
    \end{equation}  
    Next, we remark that
    $$ \displaystyle\int_{Q}\sum_{i=1}^{N}\varPhi_{x,y}^i\left(\dfrac{|u(x)- u(y)|}{||u|||x-y|^{s_i}}\right) \dfrac{dxdy}{|x-y|^N}\leqslant 1.$$   
Using the above relation, we obtain for $i=1,...,N$ :   
         $$ \displaystyle\int_{Q}\varPhi_{x,y}^i\left(\dfrac{|u(x)- u(y)|}{||u|||x-y|^{s_i}}\right) \dfrac{dxdy}{|x-y|^N}\leqslant 1.$$ 
         So, for $i=1,...,N$ we have  $[u]_{i}\leqslant ||u||$, this implies that
         \begin{equation}\label{eq2}
         ||u||_{\overrightarrow{\varPhi}}\leqslant N||u|| \text{ for all } u\in W^sL_\varPhi(\Omega).
         \end{equation}
         On the other hand, by relation $(\ref{3.})$, for $i=1,...,N$ we have 
         $$\varPhi^i_{x,y}(2t)\geqslant 2\varPhi^i_{x,y}(t) \text{ for all } t>0.$$
         Thus, for $i=1,...,N$ we deduce that 
         $$2\varPhi_{x,y}^i\left(\dfrac{|u(x)- u(y)|}{||u||_{max}|x-y|^{s_i}}\right)  \leqslant \varPhi_{x,y}^i\left(\dfrac{|u(x)- u(y)|}{2||u||_{max}|x-y|^{s_i}}\right)  \text{ for all } u\in W_0^{\overrightarrow{s}}L_{\overrightarrow{\varPhi}}(Q),$$ 
                It following that :
          $$
              \begin{aligned} 
   \displaystyle\int_{Q}\sum_{i=1}^{N}\varPhi_{x,y}^i\left(\dfrac{|u(x)-u(y)|}{2||u||_{\max}|x-y|^{s_i}}\right)\dfrac{dxdy}{|x-y|^N}
   & \leqslant\dfrac{1}{2} \left(\displaystyle\int_{Q}\sum_{i=1}^{N}\varPhi_{x,y}^i\left(\dfrac{|u(x)-u(y)|}{||u||_{\max}|x-y|^{s_i}}\right)\dfrac{dxdy}{|x-y|^N}\right) \\
   & \leqslant \dfrac{1}{2} \left(\displaystyle\int_{Q}\sum_{i=1}^{N}\varPhi_{x,y}^i\left(\dfrac{|u(x)-u(y)|}{[u]_i|x-y|^{s_i}}\right)\dfrac{dxdy}{|x-y|^N}  \right) \\
   &\leqslant  \dfrac{N}{2}.
          \end{aligned} 
              $$ 
              This last, implies that 
              $$ ||u||\leqslant N||u||_{max}$$
         Then, by the above relation and $(\ref{eq0})$, we have
         \begin{equation}\label{eq3}
         ||u||\leqslant N||u||_{max}\leqslant N||u||_{\overrightarrow{\varPhi}}.
         \end{equation}
         By relations $(\ref{eq0})$, $(\ref{eq2})$ and $(\ref{eq3})$, we deduce that Proposition $\ref{norm}$ hold true.       \end{proof}
           \begin{prop}\label{mod}
                  Assume that \hyperref[v1]{$(\varPhi_1)$} is satisfied. Then, for any $u \in W_0^{\overrightarrow{s}}L_{\overrightarrow{\varPhi}}(Q)$, the following relations hold true:
                    \begin{equation}\label{mod1}
              ||u||>1\Longrightarrow      ||u||^{\varphi_{\min}^-} \leqslant  \varPsi(u)\leqslant  ||u||^{\varphi_{\max}^+},
                    \end{equation}
                    \begin{equation}\label{mod2}
                         ||u||<1\Longrightarrow    ||u||^{\varphi_{\max}^+} \leqslant  \varPsi(u)\leqslant  ||u||^{\varphi_{\min}^-}. \end{equation}
                    \end{prop} 
                             \begin{proof} 
                       First, we show that if $||u||>1$, then $\varPsi(u) \leqslant ||u||^{\varphi_{\max}^+}$. Indeed, let $u\in W_0^{\overrightarrow{s}}L_{\overrightarrow{\varPhi}}(Q)$ such that $||u||>1$. Using the definition of the Luxemburg norm and the relation $(\ref{r10})$, we get 
                                          \begingroup\makeatletter\def\f@size{9}\check@mathfonts  $$
                                              \begin{aligned}
     \Psi(u)&=\displaystyle\int_{Q}\sum_{i=1}^{N}\varPhi_{x,y}^i\left(\dfrac{||u|||u(x)- u(y)|}{||u|||x-y|^{s_i}}\right) \dfrac{dxdy}{|x-y|^N}\\
     &\leqslant \displaystyle\sum_{i=1}^{N}||u||^{\varphi_i^+}\int_{Q}\varPhi_{x,y}^i\left(\dfrac{|u(x)- u(y)|}{||u|||x-y|^{s_i}}\right) \dfrac{dxdy}{|x-y|^N}\\
                      &\leqslant ||u||^{\varphi_{\max}^+}\displaystyle\sum_{i=1}^{N}\int_{Q}\varPhi_{x,y}^i\left(\dfrac{|u(x)- u(y)|}{||u|||x-y|^{s_i}}\right) \dfrac{dxdy}{|x-y|^N}\\
                         &\leqslant ||u||^{\varphi_{\max}^+}.
                                              \end{aligned}
                                              $$\endgroup
                
                         Next, assume that $||u||>1$. Let $\beta\in (1,||u||)$, by  $(\ref{3.})$, we have
               \begingroup\makeatletter\def\f@size{9}\check@mathfonts   $$
               \begin{aligned}
      \displaystyle\int_{Q}\sum_{i=1}^{N}\varPhi_{x,y}^i\left(\dfrac{|u(x)- u(y)|}{|x-y|^{s_i}}\right) \dfrac{dxdy}{|x-y|^N}
             &  \geqslant \displaystyle\sum_{i=1}^{N}\beta^{\varphi_i^-}\int_{Q}\varPhi_{x,y}^i\left(\dfrac{|u(x)- u(y)|}{\beta|x-y|^{s_i}}\right) \dfrac{dxdy}{|x-y|^N}\\
               & \geqslant \beta^{\varphi_{\min}^-}\displaystyle\sum_{i=1}^{N}\int_{Q}\varPhi_{x,y}^i\left(\dfrac{|u(x)- u(y)|}{\beta|x-y|^{s_i}}\right) \dfrac{dxdy}{|x-y|^N}.
               \end{aligned}   $$\endgroup
               Since $\beta< ||u||$, we find 
            $$ \displaystyle\sum_{i=1}^{N}\int_{Q}\varPhi_{x,y}^i\left(\dfrac{|u(x)- u(y)|}{\beta|x-y|^{s_i}}\right) \dfrac{dxdy}{|x-y|^N}>1.$$   
            Thus, we have
           $$ \displaystyle\sum_{i=1}^{N}\int_{Q}\varPhi_{x,y}^i\left(\dfrac{|u(x)- u(y)|}{|x-y|^{s_i}}\right) \dfrac{dxdy}{|x-y|^N} \geqslant  \beta^{\varphi_{\min}^-}.$$
           Letting $\beta \nearrow ||u||$, we deduce that $(\ref{mod1})$ holds  true.
           
             Next, we show that $\varPsi(u) \leqslant ||u||^{\varphi_{\min}^-} \text{   for all  }  u\in W_0^{\overrightarrow{s}}L_{\overrightarrow{\varPhi}}(Q) \text{ with }||u||<1$. 
                            Using the definition of the Luxemburg norm and $(\ref{r11})$, we obtain
                                 \begingroup\makeatletter\def\f@size{10}\check@mathfonts $$
                               \begin{aligned}
    \Psi(u)&\leqslant \sum_{i=1}^{N}||u||^{\varphi_i^-}\int_{Q}\varPhi_{x,y}^i\left(\dfrac{|u(x)- u(y)|}{||u|||x-y|^{s_i}}\right) \dfrac{dxdy}{|x-y|^N}\\
            &\leqslant ||u||^{\varphi_{\min}^-}\sum_{i=1}^{N}\int_{Q}\varPhi_{x,y}^i\left(\dfrac{|u(x)- u(y)|}{||u|||x-y|^{s_i}}\right) \dfrac{dxdy}{|x-y|^N}\\
                               &\leqslant||u||^{\varphi_{\min}^-}.
                               \end{aligned}
                               $$\endgroup

           Let $\xi\in (0,||u||)$. From $(\ref{3.2})$, it follows that 
           {\small \begin{equation}\label{re1}
                \begin{aligned}
               \Psi(u)&= \displaystyle\int_{Q}\sum_{i=1}^{N}\varPhi_{x,y}^i\left(\dfrac{|u(x)- u(y)|}{|x-y|^{s_i}}\right) \dfrac{dxdy}{|x-y|^N}\\
                & \geqslant   \displaystyle\sum_{i=1}^{N}\int_{Q}\xi^{\varphi_i^+}\varPhi_{x,y}^i\left(\dfrac{|u(x)- u(y)|}{\xi|x-y|^{s_i}}\right) \dfrac{dxdy}{|x-y|^N}\\
                 & \geqslant \xi^{\varphi_{\max}^+}\displaystyle\sum_{i=1}^{N}\int_{Q}\varPhi_{x,y}^i\left(\dfrac{|u(x)- u(y)|}{\xi|x-y|^{s_i}}\right) \dfrac{dxdy}{|x-y|^N} .
                \end{aligned}    \end{equation} }
                Defining     $v(x)=\dfrac{u(x)}{\xi}$ for all $x\in \Omega$. Then, $||v||=\dfrac{||u||}{\xi}>1$. Using relation $(\ref{mod1})$, we find 
                 \begin{equation}\label{re2}
                 \displaystyle\int_{Q}\sum_{i=1}^{N}\varPhi_{x,y}^i\left(\dfrac{|v(x)- v(y)|}{|x-y|^{s_i}}\right) \dfrac{dxdy}{|x-y|^N}\geqslant ||v||^{\varphi_{\min}^-}>1. \end{equation}
                 Combining $(\ref{re1})$ and $(\ref{re2})$, we deduce that
                 $$\displaystyle\int_{Q}\sum_{i=1}^{N}\varPhi_{x,y}^i\left(\dfrac{|v(x)- v(y)|}{|x-y|^{s_i}}\right) \dfrac{dxdy}{|x-y|^N}\geqslant \xi^{\varphi_{\max}^+}.$$
               Letting $\xi\nearrow ||u||$ in the above inequality, we obtain that relation $(\ref{mod2})$ holds true.
         \end{proof}
   Finally, the proof of our main results is based on the following  mountain pass theorem and Ekeland's variational principle theorem.
   \begin{thm}\cite{110}\label{an2.2}
         Let $X$ be a real Banach space and $I \in C^1(X,\R)$ with $I(0)=0$. Suppose that the following conditions hold:
         
         $(G_1)$ \label{G1} There exist
         $\rho>0 \text{  and } r>0 \text{ such that } I(u)\geqslant r \text{ for } ||u||=\rho$.
          
         $(G_2)$ \label{G2} There exists
         $e \in X \text{ with } ||e||>\rho \text{  such that } I(e)\leqslant 0$.\\
         Let
         $$c:=\inf_{\gamma \in \Gamma} \max_{t\in [0,1]}I (\gamma (t)) \text{ with } \Gamma=\left\lbrace \gamma \in C([0,1],X); \gamma(0)=0 , \gamma(1)=e \right\rbrace.$$
         Then there exists a sequence $\left\lbrace u_n\right\rbrace $ in $X$ such that 
         $$I(u_n)\rightarrow c \text{ \hspace{0.2cm} \text{and} \hspace{0.2cm}} I'(u_n)\rightarrow 0.$$
        
         \end{thm}
    \begin{thm}\label{anek}(see : \cite{ek})
    Let V be a complete metric space and $F : V \longrightarrow \R\cup \left\lbrace +\infty\right\rbrace$ be a lower semicontinuous functional on $V$, that is bounded below and not identically equal to $+\infty$. Fix $\varepsilon>0$ and a  point $u\in V$ 
      such that
     $$F(u)\leqslant \varepsilon +\inf\limits_{x\in V}F(x).$$ Then for every $\gamma > 0$,
      there exists some point $v\in V$ such that :
      $$F(v)\leqslant F(u),$$
      $$d(u,v)\leqslant \gamma$$
      and for all $w\neq v$
      $$F(w)> F(v)-\dfrac{\varepsilon}{\gamma}d(v,w).$$
    \end{thm}     
   \section{Main results and proofs}
   To simplify the notation, we ask
         $$D^{s_i}u:=\dfrac{u(x)-u(y)}{|x-y|^{s_i}}~~ \text{ and }~~ d\mu=\dfrac{dxdy}{|x-y|^N} ~~ \forall (x,y)\in Q. $$ 
         
 The dual space of $\left(W_0^{\overrightarrow{s}}L_{\overrightarrow{\varPhi}}(Q)\right) $  is denoted by $\left((W_0^{\overrightarrow{s}}L_{\overrightarrow{\varPhi}}(Q))^*,||.||_{\overrightarrow{\varPhi},*}\right) $.

 \begin{defini}
 We say that $\lambda\in \R$ is an eigenvalue of Problem \hyperref[P]{$(P_a)$} if there exists $u\in W_0^{\overrightarrow{s}}L_{\overrightarrow{\varPhi}}(Q)\setminus \left\lbrace 0\right\rbrace$ such that 
 $$M\left(\Psi(u)\right)\int_{Q} \sum_{i=1}^{N}a_{x,y}^i(|D^{s_i}u|)  D^{s_i}u D^{s_i}vd\mu-\lambda\int_{\Omega}|u|^{q(x)-2}uvdx=0$$
 for all $v\in W_0^{\overrightarrow{s}}L_{\overrightarrow{\varPhi}}(Q)$.
 \end{defini}
 
  We point that if $\lambda$ is an eigenvalue of Problem \hyperref[P]{$(P_a)$} then the corresponding $u\in W_0^{\overrightarrow{s}}L_{\overrightarrow{\varPhi}}(Q)\setminus\left\lbrace 0\right\rbrace $ is a weak solution of \hyperref[P]{$(P_a)$}.\\

  The main results in this paper are given by the following theorems.
  \begin{thm}\label{anth1}
  Assume that the function $q\in C(\overline{\Omega})$ verifies the hypothesis
  \begin{equation}
  \varphi^+_{\max}<q^-.
  \end{equation}
  Then for all $\lambda>0$ is an eigenvalue of Problem \hyperref[P]{$(P_a)$}.
  \end{thm}  
   \begin{thm}\label{anth2}
    Assume that the function $q\in C(\overline{\Omega})$ verifies the hypothesis
    \begin{equation}\label{an15}
    q^-<\varphi^-_{\min}.
    \end{equation}
    Then there exists $\lambda_*>0$ such that for any $\lambda\in (0,\lambda_*)$ is an eigenvalue of Problem \hyperref[P]{$(P_a)$}.
    \end{thm}
   
     \subsection{Auxiliary results}
  In order to prove our main results, we introduce the following functionals $J, I, J_1, I_1 : W_0^{\overrightarrow{s}}L_{\overrightarrow{\varPhi}}(Q) \longrightarrow \R$ by

  $J(u)=\widehat{M}\left( \displaystyle\int_{Q}\sum_{i=1}^{N}\varPhi_{x,y}^i\left(|D^{s_i}u| \right)d\mu\right)$

$I(u)=\displaystyle\int_{\Omega}\dfrac{1}{q(x)}|u|^{q(x)}dx$

$J_1(u)=M\left( \displaystyle\int_{Q}\sum_{i=1}^{N}\varPhi_{x,y}^i\left(|D^{s_i}u| \right)d\mu\right)\displaystyle\int_{Q} \sum_{i=1}^{N}a_{x,y}^i(|D^{s_i}u|)|D^{s_i}u|^2d\mu$

$I_1(u)=\displaystyle\int_{\Omega}|u|^{q(x)}dx$.

By a standard argument to  \cite{sr5} and \cite{3}, we have $J,I\in C^1(W_0^{\overrightarrow{s}}L_{\overrightarrow{\varPhi}}(Q),\R)$,
$$\left\langle J'(u),v\right\rangle =M\left( \displaystyle\int_{Q}\sum_{i=1}^{N}\varPhi_{x,y}^i\left(|D^{s_i}u| \right)d\mu\right)\int_{Q} \sum_{i=1}^{N}a_{x,y}^i(|D^{s_i}u|)  D^{s_i}u D^{s_i}vd\mu$$
 and 
$$\left\langle I'(u),v\right\rangle =\displaystyle\int_{\Omega}|u|^{q(x)-2}uvdx,$$
for all $u,v\in W_0^{\overrightarrow{s}}L_{\overrightarrow{\varPhi}}(Q)$.\\

Next, for each $\lambda\in \R$, we define the energetic function associated with Problem \hyperref[P]{$(P_a)$}, 
 $T_\lambda : W_0^{\overrightarrow{s}}L_{\overrightarrow{\varPhi}}(Q)\longrightarrow \R $ by 
   $$T_\lambda(u)=J(u)-\lambda I(u).$$
  Clearly, $T_\lambda \in C^1(W_0^{\overrightarrow{s}}L_{\overrightarrow{\varPhi}}(Q), \R)$ and 
  $$\left\langle T'_\lambda (u),v\right\rangle = \left\langle J' (u),v\right\rangle-\lambda\left\langle I' (u),v\right\rangle$$
  for all $u,v\in W_0^{\overrightarrow{s}}L_{\overrightarrow{\varPhi}}(Q)$. 
  \begin{lem}\label{anlem1}
  Assume that the hypothesis of Theorem $\ref{anth1}$ is fulfilled. Then, there exist $\eta>0$ and $\alpha>0$, such that $T_\lambda(u)\geqslant \alpha>0$ for any $u\in W_0^{\overrightarrow{s}}L_{\overrightarrow{\varPhi}}(Q)$ with $\|u\|_{\overrightarrow{\varPhi}}=\eta$.
 \end{lem}
 \begin{proof}
 First, we point out that
 \begin{equation*}
 |u(x)|^{q^-}+|u(x)|^{q^+}\geqslant |u(x)|^{q(x)}~~\forall x\in \overline{\Omega}.
 \end{equation*}
 Using the above inequality and the definition of $T_\lambda$, we find that 
 \begin{equation}\label{an3}
 T_\lambda(u)\geqslant m_0\int_{Q}\sum_{i=1}^{N}\varPhi_{x,y}^i\left(|D^{s_i}u| \right)d\mu-\dfrac{\lambda}{q^-}\left( \|u\|_{q^-}^{q^-}+\|u\|_{q^+}^{q^+}\right) 
 \end{equation}
 for any $u\in W_0^{\overrightarrow{s}}L_{\overrightarrow{\varPhi}}(Q)$. Since $W_0^{\overrightarrow{s}}L_{\overrightarrow{\varPhi}}(Q)$ is continuously embedded in $L^{q^{+}}(\Omega)$ and in $L^{q^{-}}(\Omega)$ it follows that there exist two positive constants $c_1$ and $c_2$ such that
 \begin{equation}\label{an24}
 ||u||_{\overrightarrow{\varPhi}}\geqslant c_1||u||_{q^+} ~~\forall u\in W_0^{\overrightarrow{s}}L_{\overrightarrow{\varPhi}}(Q)\end{equation}
 and 
 \begin{equation}\label{an25}
 ||u||_{\overrightarrow{\varPhi}}\geqslant c_2||u||_{q^-} ~~\forall u\in W_0^{\overrightarrow{s}}L_{\overrightarrow{\varPhi}}(Q).\end{equation}
 Next, for $u\in W_0^{\overrightarrow{s}}L_{\overrightarrow{\varPhi}}(Q)$ with $\|u\|_{\overrightarrow{\varPhi}}<1$, so, we have $[u]_{i}<1$ for any 
 $i=1,...,N$. Then
 \begin{equation}\label{an4}
 \begin{aligned}
 \dfrac{\|u\|^{\varphi^+_{\max}}_{\overrightarrow{\varPhi}}}{N^{\varphi_{\max}^+-1}} & =N\left( \sum_{i=1}^{N}\dfrac{1}{N}[u]_{i}\right)^{\varphi^+_{\max}}\\
 & \leqslant \sum_{i=1}^{N}[u]_{i}^{\varphi^+_{\max}}\\
  & \leqslant \sum_{i=1}^{N}[u]_{i}^{\varphi^+_{i}}\\
  &\leqslant \int_{Q}\sum_{i=1}^{N}\varPhi_{x,y}^i\left(|D^{s_i}u| \right)d\mu.
  \end{aligned}
 \end{equation}
 Relations $(\ref{an3})-(\ref{an4})$, imply that
$$ 
\begin{aligned}
T_\lambda(u)& \geqslant \dfrac{m_0\|u\|^{\varphi^+_{\max}}_{\overrightarrow{\varPhi}}}{N^{\varphi_{\max}^+-1}}-\dfrac{\lambda}{q^-}\left( (c_1\|u\|_{\overrightarrow{\varPhi}})^{q^-}+(c_2\|u\|_{\overrightarrow{\varPhi}})^{q^+}\right)\\
& = \left( c_3-c_4 \|u\|^{q^+-\varphi^+_{\max}}_{\overrightarrow{\varPhi}}-c_5 \|u\|^{q^--\varphi^+_{\max}}_{\overrightarrow{\varPhi}}\right) \|u\|^{\varphi^+_{\max}}_{\overrightarrow{\varPhi}} 
\end{aligned}
$$
for any $u\in W_0^{\overrightarrow{s}}L_{\overrightarrow{\varPhi}}(Q)$ with $\|u\|_{\overrightarrow{\varPhi}}<1$, where $c_3, c_4$, and $c_5$ are positive constants. Since the function $g : [0,1]\longrightarrow \R$ defined by:
$$g(t)=c_3-c_4 t^{q^+-\varphi^+_{\max}}-c_5 t^{q^--\varphi^+_{\max}}$$
is positive in a neighborhood of the origin, the conclusion of the lemma follows at once. 
\end{proof}
  \begin{lem}\label{anlem2}
    Assume that the hypothesis of Theorem $\ref{anth1}$ is fulfilled. Then, there exist $e>0$ with $\|e\|_{\overrightarrow{\varPhi}}=\eta$ (where $\eta$ is given by Lemma $\ref{anlem1}$) such that $T_\lambda(e)<0$.
   \end{lem}
   \begin{proof}
   Let $\theta\in C_0^\infty(\Omega)$, $\theta\geqslant 0$ and $\theta\neq 0$, be fixed, and let $t>1$. Using Lemma $\ref{2.2..}$, we find that 
   $$
   \begin{aligned}
  T_\lambda(t\theta) &=\widehat{M}(1)\left( \displaystyle\int_{Q}\sum_{i=1}^{N}\varPhi_{x,y}^i\left(t|D^{s_i}\theta| \right)d\mu\right)^{\frac{1}{\theta}} -\lambda\displaystyle\int_{\Omega}\dfrac{1}{q(x)}t^{q(x)}|\theta|^{q(x)}dx\\
  &\leqslant \widehat{M}(1)\left(\displaystyle\int_{Q}\sum_{i=1}^{N}t^{\varphi_i^+}\varPhi_{x,y}^i\left(|D^{s_i}\theta| \right)d\mu\right)^{\frac{1}{\theta}}-\lambda\displaystyle\int_{\Omega}\dfrac{t^{q^-}}{q^+}|\theta|^{q(x)}dx\\
    &\leqslant \widehat{M}(1)t^{\frac{\varphi_{\max}^+}{\theta}}\left( \displaystyle\int_{Q}\sum_{i=1}^{N}\varPhi_{x,y}^i\left(|D^{s_i}\theta| \right)d\mu\right)^{\frac{1}{\theta}}-\dfrac{\lambda t^{q^-}}{q^+}\displaystyle\int_{\Omega}|\theta|^{q(x)}dx.
    \end{aligned}
   $$
   Since $q^->\frac{\varphi_{\max}^+}{\theta}$, it is clear that $\lim\limits_{t\rightarrow \infty}T_\lambda(t\theta)=-\infty$. Then, for $t>1$ large enough, we can take $e=t \theta$ such that $\|e\|_{\overrightarrow{\varPhi}}=\eta$ and  $T_\lambda(e)<0$. This completes the proof.
      \end{proof}
  \begin{lem}\label{anlem3}
Assume that the hypothesis of Theorem $\ref{anth2}$ is fulfilled. Then, there exists $\lambda_*>0$ such that for any
$\lambda\in (0,\lambda_*)$, there are $\rho, \alpha>0$, such that $T_\lambda(u)\geqslant \alpha>0$ for any $u\in W_0^{\overrightarrow{s}}L_{\overrightarrow{\varPhi}}(Q)$ with $||u||_{\overrightarrow{\varPhi}}=\rho$.
  \end{lem}  
  \begin{proof}
  Since $W_0^{\overrightarrow{s}}L_{\overrightarrow{\varPhi}}(Q)$ is continuously embedded in $L^{q(x)}(\Omega)$, it follows that there exists a positive constant $c_1$  such that
   \begin{equation}\label{an11}
  ||u||_{\overrightarrow{\varPhi}}\geqslant c_1||u||_{q(x)} ~~\forall u\in W_0^{\overrightarrow{s}}L_{\overrightarrow{\varPhi}}(Q)\end{equation} 
   we fix $\rho \in (0,1)$ such that $\rho<\dfrac{1}{c_1}$. Then relation $(\ref{an11})$ implies that 
   $$\|u\|_{q(x)}<1~~\text{for all } u\in W_0^{\overrightarrow{s}}L_{\overrightarrow{\varPhi}}(Q) ~~\text{with } ||u||_{\overrightarrow{\varPhi}}=\rho.$$
   Then, we can apply Proposition $\ref{anproop5}$, and we have
   \begin{equation}\label{an12}
   \int_\Omega |u(x)|^{q(x)}dx\leqslant \|u\|_{q(x)}^{q^-}~~\text{ for all } u\in W_0^{\overrightarrow{s}}L_{\overrightarrow{\varPhi}}(Q)~~ \text{ with } ||u||_{\overrightarrow{\varPhi}}=\rho.
   \end{equation}
   Relation $(\ref{an11})$  and $(\ref{an12})$ implies that
    \begin{equation}\label{an13}
      \int_\Omega |u(x)|^{q(x)}dx\leqslant c_1^{q^-}\|u\|_{\overrightarrow{\varPhi}}^{q^-}~~\text{ for all } u\in W_0^{\overrightarrow{s}}L_{\overrightarrow{\varPhi}}(Q)~~ \text{ with } ||u||_{\overrightarrow{\varPhi}}=\rho.
      \end{equation}
   Taking into account Relations $(\ref{an4})$ and $(\ref{an13})$, we deduce that for any $u \in W_0^{\overrightarrow{s}}L_{\overrightarrow{\varPhi}}(Q)$ with  $||u||_{\overrightarrow{\varPhi}}=\rho$, the following inequalities hold true:
   $$
   \begin{aligned}
   T_\lambda(u) &\geqslant  \dfrac{m_0\|u\|^{\varphi^+_{\max}}_{\overrightarrow{\varPhi}}}{N^{\varphi_{\max}^+-1}}-\dfrac{\lambda}{q^-}\int_\Omega|u(x)|^{q(x)}dx\\
   & \geqslant \dfrac{m_0\|u\|^{\varphi^+_{\max}}_{\overrightarrow{\varPhi}}}{N^{\varphi_{\max}^+-1}}-\dfrac{\lambda c_1^{q-}}{q^-}\|u\|^{q^-}_{\overrightarrow{\varPhi}}\\
  & =\rho^{q^-}\left(m_0\dfrac{\rho^{\varphi^+_{\max}-q^-}}{N^{\varphi_{\max}^+-1}}- \dfrac{\lambda c_1^{q-}}{q^-}\right).
   \end{aligned}
   $$
   Hence, if we define
  \begin{equation}\label{an40}
  \lambda_*=m_0\dfrac{\rho^{\varphi^+_{\max}-q^-}}{2c_1^{q^-} N^{\varphi_{\max}^+-1}}q^-.\end{equation}
  Then, for any $\lambda\in (0,\lambda_*)$ and $u\in W_0^{\overrightarrow{s}}L_{\overrightarrow{\varPhi}}(Q)$ with  $||u||_{\overrightarrow{\varPhi}}=\rho$, we have
  $$T_\lambda(u)\geqslant \alpha>0,$$
  such that 
  $$\alpha=\dfrac{m_0\rho^{\varphi^+_{\max}}}{2 N^{\varphi^+_{\max}-1}}.$$
  This completes the proof.
     \end{proof} 
   \begin{lem}\label{anlem4}
   Assume that the hypothesis of Theorem $\ref{anth2}$ is fulfilled. Then, there exists $\phi>0$ such that $\phi\geqslant 0$,  $\phi\neq 0$, and $T_\lambda(t\phi)<0$ for $t>0$ small enough.
     \end{lem}    
   \begin{proof}
 By assumption   $(\ref{an15})$ we can chose $\varepsilon_0>0$ such that $q^-+\varepsilon_0<\varphi^-_{\min}$. On the other hand, since $q\in C(\overline{\Omega})$, it follows that there exists an open set $\Omega_0\subset \Omega$ such that $|q(x)-q^-|<\varepsilon_0$ for all $x\in \Omega_0$. Thus, $q(x)\leqslant q^-+\varepsilon_0<\varphi^-_{\min}$ for all $x\in \Omega_0$. 
 Let $\phi\in C_0^\infty(\Omega)$ be such that $supp(\phi)\supset \overline{\Omega_0}$, $\phi(x)=1$ for all $x\in \overline{\Omega_0}$, and $0\leqslant \phi\leqslant 1$ in $\overline{\Omega_0}$. Then, using the above information and the definition of $\varphi_i^-$, for any $t\in(0,1)$, we have
 $$
 \begin{aligned}
 T_\lambda(t\phi)& =\widehat{M}(1)\left(\displaystyle\int_{Q}\sum_{i=1}^{N}\varPhi_{x,y}^i\left(t|D^{s_i}\phi| \right)d\mu\right)^{\frac{1}{\theta}} -\lambda\displaystyle\int_{\Omega}\dfrac{1}{q(x)}t^{q(x)}|\phi|^{q(x)}dx\\
 &\leqslant \widehat{M}(1)\left(\displaystyle\int_{Q}\sum_{i=1}^{N}t^{\varphi_i^-}\varPhi_{x,y}^i\left(|D^{s_i}\phi| \right)d\mu\right)^{\frac{1}{\theta}}-\lambda\displaystyle\int_{\Omega_0}\dfrac{t^{q(x)}}{q(x)}|\phi|^{q(x)}dx\\
     &\leqslant t^{\frac{\varphi_{\min}^-}{\theta}}\widehat{M}(1)\left(\displaystyle\int_{Q}\sum_{i=1}^{N}\varPhi_{x,y}^i\left(|D^{s_i}\phi| \right)d\mu\right)^{\frac{1}{\theta}}-\dfrac{\lambda t^{q^-+\varepsilon_0}}{q^+}\displaystyle\int_{\Omega_0}|\phi|^{q(x)}dx.
     \end{aligned}
    $$
    Therefore $T_\lambda(t\phi)<0$, for $t<\delta^{1/(\frac{\varphi^-_{\min}}{\theta}-q^--\varepsilon_0)}$ with
    $$0<\delta<\min\left\lbrace 1,~~ \dfrac{\dfrac{\lambda}{q^+}\displaystyle\int_{\Omega_0}|\phi|^{q(x)}dx}{\left( \displaystyle\int_{Q}\sum_{i=1}^{N}\varPhi_{x,y}^i\left(|D^{s_i}\phi| \right)d\mu\right)^{\frac{1}{\theta}}} \right\rbrace. $$
   This is possible, since we claim that 
   $$  \left( \displaystyle\int_{Q}\sum_{i=1}^{N}\varPhi_{x,y}^i\left(|D^{s_i}\phi| \right)d\mu\right)^{\frac{1}{\theta}}>0.$$
   Indeed, it is clear that
   $$\int_{\Omega_0}|\phi|^{q(x)}dx\leqslant \int_{\Omega}|\phi|^{q(x)}dx\leqslant \int_{\Omega}|\phi|^{q^-}dx.$$
   On the other hand, since  $W_0^{\overrightarrow{s}}L_{\overrightarrow{\varPhi}}(Q)$ is continuously embedded in $L^{q^-}(\Omega)$, it follows that there exists a positive constant $c$  such that 
   $$|\phi|_{q^-}\leqslant c ||\phi||_{\overrightarrow{\varPhi}}.$$
   The last two inequalities imply that
   $$\|\phi\|_{\overrightarrow{\varPhi}}>0$$
   and combining this fact with Proposition $\ref{mod}$, the claim follows at once. The proof of the lemma is now completed.
   \end{proof}   
  \subsection{Proof of mains results}
  \begin{proof}[\textbf{Proof of Theorem $\ref{anth1}$}]
  By Lemmas $\ref{anlem1}$ and $\ref{anlem2}$, we can using the mountain pass theorem $\ref{an2.2}$, and we deduce the existence of a sequence $\left\lbrace u_n\right\rbrace \in W_0^{\overrightarrow{s}}L_{\overrightarrow{\varPhi}}(Q)$ such that
  \begin{equation}\label{an29}
  T_\lambda(u_n)\rightarrow c_1~~\text{and}~~ T'_\lambda(u_n)\rightarrow 0~~\text{in}~~(W_0^{\overrightarrow{s}}L_{\overrightarrow{\varPhi}}(Q))^*~~\text{as}~~n\rightarrow \infty.
  \end{equation}
 We assume that $\left\lbrace u_n\right\rbrace$ is bounded in  $W_0^{\overrightarrow{s}}L_{\overrightarrow{\varPhi}}(Q)$. By contradiction, we suppose that there exists a subsequence still denoted by $\left\lbrace u_n\right\rbrace $ such that $\|u_n\|_{\overrightarrow{\varPhi}}\rightarrow \infty$ and that $\|u_n\|_{\overrightarrow{\varPhi}}>1$ for all $n$. Relation $(\ref{an29})$ and the above considerations implies that for $n$ large enough, we have
{\small $$
 \begin{aligned}
 1&+c_1+\|u_n\|_{\overrightarrow{\varPhi}} \geqslant T_\lambda(u_n)-\dfrac{1}{q^-}\left\langle T'_\lambda(u_n), u_n\right\rangle\\
 &\geqslant \widehat{M}\left(\sum_{i=1}^{N} \int_{Q} \varPhi_{x,y}^i\left(|D^{s_i}u_n| \right)\right)-\dfrac{1}{q^-}M\left(\sum_{i=1}^{N} \int_{Q} \varPhi_{x,y}^i\left(|D^{s_i}u_n| \right)\right)\sum_{i=1}^{N} \int_{Q}\varphi_{x,y}^i(|D^{s_i}u_n|)D^{s_i}u_nd\mu  \\
 &\geqslant \widehat{M}\left(\sum_{i=1}^{N} \int_{Q} \varPhi_{x,y}^i\left(|D^{s_i}u_n| \right)d\mu\right)-\dfrac{\varphi^+_{\max}}{q^-}M\left(\sum_{i=1}^{N} \int_{Q} \varPhi_{x,y}^i\left(|D^{s_i}u_n| \right)d\mu\right)\sum_{i=1}^{N} \int_{Q} \varPhi_{x,y}^i\left(|D^{s_i}u_n| \right)d\mu  \\
 &\geqslant \left( 1-\dfrac{\varphi^+_{\max}}{q^-\theta}\right) \widehat{M}\left( \sum_{i=1}^{N} \int_{Q} \varPhi_{x,y}^i\left(|D^{s_i}u_n| \right)d\mu\right)\\
 &\geqslant
 m_0\left( 1-\dfrac{\varphi^+_{\max}}{q^-\theta}\right) \left( \dfrac{\|u_n\|^{\varphi^-_{\min}}_{\overrightarrow{\varPhi}}}{N^{\varphi_{\min}^--1}}-N\right).
  \end{aligned}
 $$}
 Since $q^-\theta>\varphi^+_{\max}$, and dividing by $\|u_n\|^{\varphi^-_{\min}}_{\overrightarrow{\varPhi}}$ in the above inequality and passing to the limit as $n\rightarrow \infty$, we obtain a contradiction. Then $\left\lbrace u_n\right\rbrace $ is bounded in $W_0^{\overrightarrow{s}}L_{\overrightarrow{\varPhi}}(Q)$. This information combined with the fact that $W_0^{\overrightarrow{s}}L_{\overrightarrow{\varPhi}}(Q)$ is reflexive, implies that there exists a subsequence still denoted by $\left\lbrace u_n\right\rbrace $ and $u_0\in W_0^{\overrightarrow{s}}L_{\overrightarrow{\varPhi}}(Q)$ such that $\left\lbrace u_n\right\rbrace $ converges weakly to $u_0$ in $W_0^{\overrightarrow{s}}L_{\overrightarrow{\varPhi}}(Q)$. On the other hand, since $W_0^{\overrightarrow{s}}L_{\overrightarrow{\varPhi}}(Q)$ is compactly embedded in $L^{q(x)}(\Omega)$, it follows that $\left\lbrace u_n\right\rbrace $ converges strongly to $u_0$ in $L^{q(x)}(\Omega)$.
 Then by H\"{o}lder inequality, we have that 
 $$ \lim\limits_{n\rightarrow \infty}\int_\Omega |u_n|^{q(x)-2}u_n(u_n-u_0)dx=0.$$
 This fact and relation $(\ref{an29})$, implies that
 $$\lim\limits_{n\rightarrow \infty}\left\langle T'_\lambda(u_n), u_n-u_0\right\rangle =0.$$
 Thus we deduce that 
 \begin{equation}\label{an311}
 \lim\limits_{n\rightarrow \infty}M\left(\Psi(u_n)\right)\sum_{i=1}^{N} \int_{Q}a_{x,y}^i(|D^{s_i}u_n|)D^su_n\left( D^{s_i}u_n-D^{s_i}u_0 \right)d\mu =0.
 \end{equation}
 Since ${u_n}$ is bounded in $W_0^{\overrightarrow{s}}L_{\overrightarrow{\varPhi}}(Q)$, then by Proposition \ref{mod}, $\Psi(u_n)$ is also bounded,
then
$$\Psi(u_n)\longrightarrow t_0~~~~\text{as}~~n\rightarrow \infty$$
If $t_0 = 0$, then using Proposition \ref{mod}, we get ${u_n}$ that strongly converges to $u_0$ in
$W_0^{\overrightarrow{s}}L_{\overrightarrow{\varPhi}}(Q)$,\\
If $t_0 > 0$, it follows from the continuity of the function $M$ that
$$M\left(\Psi(u_n)\right)\longrightarrow M(t_0)~~~~\text{as}~~n\rightarrow \infty$$
Thus, by $(M_0)$, for sufficiently large $n$, we get
$$M\left(\Psi(u_n)\right)\geqslant C_0>0,$$
This fact and relation $(\ref{an311})$ implies that 
\begin{equation}\label{an31}
 \lim\limits_{n\rightarrow \infty}\sum_{i=1}^{N} \int_{Q}a_{x,y}^i(|D^{s_i}u_n|)D^{s_i}u_n\left( D^{s_i}u_n-D^{s_i}u_0 \right)d\mu =0.
 \end{equation}
 Since $\left\lbrace u_n\right\rbrace $ converge weakly to $u_0$ in $W_0^{\overrightarrow{s}}L_{\overrightarrow{\varPhi}}(Q)$, by relation $(\ref{an31})$, we find that 
 {\small \begin{equation}\label{an32}
  \lim\limits_{n\rightarrow \infty}\sum_{i=1}^{N} \int_{Q}\left( a_{x,y}^i(|D^{s_i}u_n|)D^{s_i}u_n-a_{x,y}^i(|D^{s_i}u_0|)D^{s_i}u_0\right) \left( D^{s_i}u_n-D^{s_i}u_0 \right)d\mu =0.
 \end{equation}}
 Since, for each $i\in \left\lbrace 1,...,N\right\rbrace $, $\varPhi_{x,y}^i$ is convex, we have
 $$\varPhi_{x,y}^i(|D^{s_i}u|)\leqslant \varPhi_{x,y}^i\left( \dfrac{|D^{s_i}u+D^{s_i}v|}{2}\right) +a_{x,y}^i(|D^{s_i}u|)D^su\dfrac{D^{s_i}u-D^{s_i}v}{2}$$
  $$\varPhi_{x,y}^i(|D^{s_i}v|)\leqslant \varPhi_{x,y}^i\left( \dfrac{|D^{s_i}u+D^{s_i}v|}{2}\right) +a_{x,y}^i(|D^{s_i}v|)D^{s_i}v\dfrac{D^{s_i}v-D^{s_i}u}{2}$$
  for every $u,v \in W_0^{\overrightarrow{s}}L_{\overrightarrow{\varPhi}}(Q)$.
  Adding the above two relations and integrating over $Q$, we find that
\begin{equation}\label{an33}
  \begin{aligned}
  \dfrac{1}{2}&\int_{Q} \left( a_{x,y}^i(|D^{s_i}u|)D^{s_i}u-a_{x,y}^i(|D^{s_i}v|)D^{s_i}v\right) \left( D^{s_i}u-D^{s_i}v \right)d\mu\\
  &\geqslant \int_Q\varPhi_i(|D^{s_i}u|) d\mu+ \int_Q\varPhi_{x,y}^i(|D^{s_i}v|) d\mu- 2\int_Q\varPhi_{x,y}^i\left( \dfrac{|D^{s_i}u-D^{s_i}v|}{2}\right)d\mu,
    \end{aligned}
\end{equation}
 for every $u,v \in W_0^{\overrightarrow{s}}L_{\overrightarrow{\varPhi}}(Q)$, and each $i\in \left\lbrace 1,...,N\right\rbrace.$
 
 On the other hand, since for each $i\in \left\lbrace 1,...,N\right\rbrace $ we know that $\varPhi_{x,y}^i~:~ [0,\infty) \rightarrow \R$ is an increasing continuous function, with $\varPhi_{x,y}^i(0)=0$. Then by the conditions \hyperref[v1]{$(\varPhi_1)$} and $(\ref{r1})$, we can apply \cite[Lemma 2.1]{Lam}  in order to obtain
 \begin{equation}\label{an34}
 \begin{aligned}
 \dfrac{1}{2}&\left[ \int_Q\varPhi_{x,y}^i(|D^su|) d\mu+ \int_Q\varPhi_{x,y}^i(|D^{s_i}v|) d\mu\right] \\
& \geqslant \int_Q\varPhi_{x,y}^i\left( \dfrac{|D^{s_i}u+D^{s_i}v|}{2}\right)d\mu+\int_Q\varPhi_{x,y}^i\left( \dfrac{|D^{s_i}u-D^{s_i}v|}{2}\right)d\mu,
 \end{aligned}
 \end{equation}
 for every $u,v \in W_0^{\overrightarrow{s}}L_{\overrightarrow{\varPhi}}(Q)$, and each $i\in \left\lbrace 1,...,N\right\rbrace.$ 
 By $(\ref{an33})$ and $(\ref{an34})$, it follows that for each $i\in \left\lbrace 1,...,N\right\rbrace$, we have
 \begin{equation}\label{an35}
 \begin{aligned}
 & \int_{Q} \left( a_{x,y}^i(|D^{s_i}u|)D^{s_i}u-a_{x,y}^i(|D^{s_i}v|)D^{s_i}v\right) \left( D^su-D^sv \right)d\mu\\
 &\geqslant 4\int_Q\varPhi_{x,y}^i\left( \dfrac{|D^{s_i}u-D^{s_i}v|}{2}\right)d\mu
 \end{aligned}
 \end{equation}
 for every $u,v \in W_0^{\overrightarrow{s}}L_{\overrightarrow{\varPhi}}(Q)$, and each $i\in \left\lbrace 1,...,N\right\rbrace.$ \\
 Relations $(\ref{an32})$ and $(\ref{an35})$ show that $\left\lbrace u_n\right\rbrace $ converge strongly to $u_0$ in $W_0^{\overrightarrow{s}}L_{\overrightarrow{\varPhi}}(Q)$. Then by relation $(\ref{an29})$, we have 
 $T_\lambda(u_0)=c_1>0$ and $T'_\lambda(u_0)=0$, that is, $u_0$ is a non trivial weak solution.
    \end{proof}
  \begin{proof}[\textbf{Proof of Theorem $\ref{anth2}$}]    
   Let $\lambda_*>0$ be defined as in $(\ref{an40})$ and $\lambda\in (0,\lambda_*)$. By Lemma $\ref{anlem3}$ it follows that on the boundary oh the ball centered in the origin and of radius $\rho$ in $W_0^{\overrightarrow{s}}L_{\overrightarrow{\varPhi}}(Q)$, denoted by $B_\rho(0)$, we have 
         $$\inf\limits_{\partial B_\rho(0)}T_\lambda>0.$$
      On the other hand, by Lemma $\ref{anlem4}$, there exists $\phi \in W_0^{\overrightarrow{s}}L_{\overrightarrow{\varPhi}}(Q)$ such that $J_\lambda(t\phi)<0$ for all $t>0$ small enough. Moreover for any $u\in B_\rho(0)$, we have 
      $$
               \begin{aligned}
           T_\lambda(u)\geqslant \dfrac{m_0}{N^{\varphi_{\min}^--1}} \|u\|^{\varphi^-_{\min}}_{\overrightarrow{\varPhi}}-\dfrac{\lambda c_1^{q^-}}{q^-}\|u\|^{q^-}_{\overrightarrow{\varPhi}}.
                 \end{aligned}
                            $$
      It follows that
      $$-\infty<c:=\inf\limits_{\overline{B_\rho(0)}} T_\lambda<0.$$   
      We let now $0<\varepsilon <\inf\limits_{\partial  B_\rho(0)}  T_\lambda -  \inf\limits_{B_\rho(0)} T_\lambda.$    Applying Theorem $\ref{anek}$ to the functional 
      $T_\lambda : \overline{B_\rho(0)}\longrightarrow \R$, we find $u_\varepsilon \in \overline{B_\rho(0)}$ such that 
       $$
            \left\{ 
                 \begin{array}{clclc}
               T_\lambda(u_\varepsilon)&<\inf\limits_{\overline{B_\rho(0)}} T_\lambda+\varepsilon,& \\\\
                 T_\lambda(u_\varepsilon)&< T_\lambda(u)+\varepsilon ||u-u_\varepsilon||_{\overrightarrow{\varPhi}},& \text{  } u\neq u_\varepsilon.
                 \end{array}
                 \right. 
              $$
       Since  $T_\lambda(u_\varepsilon)\leqslant  \inf\limits_{\overline{B_\rho(0)}} T_\lambda+\varepsilon\leqslant \inf\limits_{B_\rho(0)} T_\lambda+\varepsilon < \inf\limits_{\partial  B_\rho(0)}  T_\lambda$, we deduce $u_\varepsilon  \in B_\rho(0)$. 
       
       Now, we define $\Lambda_\lambda :  \overline{B_\rho(0)}\longrightarrow \R$ by 
       $$\Lambda_\lambda(u)=T_\lambda(u)+\varepsilon||u-u_\varepsilon||_{\overrightarrow{\varPhi}}.$$
       It's clear that $u_\varepsilon$ is a minimum point of $\Lambda_\lambda$ and then
       $$\dfrac{\Lambda_\lambda(u_\varepsilon+t v)-\Lambda_\lambda(u_\varepsilon)}{t}\geqslant 0$$ 
       for small $t>0$, and any $v\in B_\rho(0).$ The above relation yields 
           $$\dfrac{T_\lambda(u_\varepsilon+t v)-T_\lambda(u_\varepsilon)}{t}+\varepsilon||v||_{\overrightarrow{\varPhi}}\geqslant 0.$$
           Letting $t\rightarrow$ it follows that $\left\langle T'_{\lambda}(u_\varepsilon),v\right\rangle +\varepsilon ||v||_{\overrightarrow{\varPhi}}>0$ and we infer that $$||T'_{\lambda}(u_\varepsilon)||_{\overrightarrow{\varPhi},*}\leqslant \varepsilon.$$
            We deduce that there exists a sequence $\left\lbrace v_n\right\rbrace \subset B_\rho(0)$ such that 
            \begin{equation}\label{an10}
            J_\lambda(v_n) \longrightarrow c \text{ and } J'_\lambda(v_n)\longrightarrow 0.
            \end{equation}
         It is clear that $\left\lbrace v_n\right\rbrace $ is bounded in $W_0^{\overrightarrow{s}}L_{\overrightarrow{\varPhi}}(Q)$. Thus, there exists $v\in W_0^{\overrightarrow{s}}L_{\overrightarrow{\varPhi}}(Q)$, such that up to a subsequence   $\left\lbrace v_n\right\rbrace $ converges weakly to $v$ in $W_0^{\overrightarrow{s}}L_{\overrightarrow{\varPhi}}(Q)$. Actually with similar arguments to those used of the end of Theorem $\ref{anth1}$, we can show that $\left\lbrace v_n\right\rbrace $ is converges strongly to $v$ in $W_0^{\overrightarrow{s}}L_{\overrightarrow{\varPhi}}(Q)$. Thus, by   $(\ref{an10})$
         $$T_\lambda(v)=c <0~~\text{ and }~~ T'_\lambda(v)=0.$$
         Then, $v$ is a nontrivial weak solution for Problem \hyperref[P]{$(P_a)$}. This complete the proof.
             \end{proof}

   \section*{Conflict of Interests} No potential conflict of interest was reported by the authors.                   
                     
   \section*{Data availability Data} Sharing not applicable to this article as no
   datasets were generated or analyzed during the current study.                
                     
  \section*{Disclosure statement}
My manuscript has no associate data.

\end{document}